\documentclass{amsart}

\usepackage{amssymb}

\newtheorem{theorem}{Theorem}[section]
\newtheorem{corollary}[theorem]{Corollary}

\newtheorem{proposition}[theorem]{Proposition}
\theoremstyle{definition}
\newtheorem{definition}[theorem]{Definition}
\newtheorem{example}[theorem]{Example}
\theoremstyle{remark}

\numberwithin{equation}{section}
\begin{document}



\title[A Note on Almost Anti-Periodic Functions in Banach Spaces]{\bf A Note on Almost Anti-Periodic Functions in Banach Spaces}

\vspace{15mm} 


\author{Marko Kosti\' c}
\address{Faculty of Technical Sciences,
University of Novi Sad,
Trg D. Obradovi\' ca 6, 21125 Novi Sad, Serbia}
\email{marco.s@verat.net}

\author{Daniel Velinov}
\address{Department for Mathematics, Faculty of Civil Engineering, Ss. Cyril and Methodius University, Skopje,
Partizanski Odredi
24, P.O. box 560, 1000 Skopje, Macedonia}
\email{velinovd@gf.ukim.edu.mk}

\keywords{Almost anti-periodic functions, Almost periodic functions, Anti-Periodic Functions, Bohr transform, Banach spaces.
\indent {\it Received}: \\
\indent {\it Revised}:}

\begin{abstract}
The main aim of this note is to introduce the notion of an almost anti-periodic function in Banach space. We prove some characterizations for this class of functions, investigating also its relationship with the classes of anti-periodic functions and
almost periodic functions in Banach spaces.
\end{abstract}

\maketitle

\section{Introduction and Preliminaries}

As mentioned in the abstract, the main aim of this note is to introduce the notion of an almost anti-periodic function in Banach space as well as to prove some characterizations for this class of functions. Any anti-periodic function is almost anti-periodic, and any almost anti-periodic function is almost periodic. Unfortunately, almost anti-periodic functions do not have a linear vector structure with the usually considered operations of pointwise addition of functions and multiplication with scalars. The main result of paper is Theorem \ref{ghf-prc}, in which we completely profile the closure of linear span of almost anti-periodic functions in the space of almost periodic functions. We also prove some other statements regarding almost anti-periodic functions, and introduce
the concepts of Stepanov almost anti-periodic functions, asymptotically almost anti-periodic functions and Stepanov asymptotically almost anti-periodic functions. We investigate the almost anti-periodic properties of convolution products, providing also
 a few elementary examples and applications.

Let $(X,\|  \cdot \|)$ be a complex Banach space. By $C_{b}([0,\infty):X)$ we denote the space consisting of all bounded continuous functions from $[0,\infty)$ into $X;$ the symbol $C_{0}([0,\infty):X)$ denotes the closed subspace of $C_{b}([0,\infty):X)$ consisting of functions vanishing at infinity.
By $BUC([0,\infty):X)$ we denote the space consisted of all bounded uniformly continuous functions from $[0,\infty)$
to $X.$ This space becomes one of Banach's endowed with the sup-norm.

The concept of almost periodicity was introduced by Danish mathematician H. Bohr around 1924-1926 and later generalized by many other authors (cf. \cite{Di}-\cite{Gu} and \cite{Le} for more details on the subject). Let
$I={\mathbb R}$ or $I=[0,\infty),$ and let $f : I \rightarrow X$ be continuous. Given $\epsilon>0,$ we call $\tau>0$ an $\epsilon$-period for $f(\cdot)$ iff
\begin{align*}
\| f(t+\tau)-f(t) \| \leq \epsilon,\quad t\in I.
\end{align*}
The set constituted of all $\epsilon$-periods for $f(\cdot)$ is denoted by $\vartheta(f,\epsilon).$ It is said that $f(\cdot)$ is almost periodic, a.p. for short, iff for each $\epsilon>0$ the set $\vartheta(f,\epsilon)$ is relatively dense in $I,$ which means that
there exists $l>0$ such that any subinterval of $I$ of length $l$ meets $\vartheta(f,\epsilon)$.

The space consisted of all almost periodic functions from the interval $I$ into $X$ will be denoted by $AP(I:X).$ Equipped with the sup-norm, $AP(I:X)$ becomes a Banach space.

For the sequel, we need some preliminary results appearing already in the pioneering paper \cite{Ba} by H. Bart and S. Goldberg, who introduced the notion of an almost periodic strongly continuous semigroup there (see \cite{Ar} for more details on the subject).
The translation semigroup $(W(t))_{t\geq 0}$ on $AP([0,\infty) : X),$ given by $[W(t)f](s):=f(t+s),$ $t\geq 0,$ $s\geq 0,$ $f\in AP([0,\infty) : X)$ is consisted solely of surjective isometries $W(t)$ ($t\geq 0$) and can be extended to a $C_{0}$-group $(W(t))_{t\in {\mathbb R}}$ of isometries on $AP([0,\infty) : X),$ where $W(-t):=W(t)^{-1}$ for $t>0.$ Furthermore, the mapping $E : AP([0,\infty) : X) \rightarrow AP({\mathbb R} : X),$ defined by
$$
[Ef](t):=[W(t)f](0),\quad t\in {\mathbb R},\ f\in AP([0,\infty) : X),
$$
is a linear surjective isometry and $Ef$ is the unique continuous almost periodic extension of a function $f(\cdot)$ from $AP([0,\infty) : X)$ to the whole real line. We have that $
[E(Bf)]=B(Ef)$ for all $B\in L(X)$ and $f\in  AP([0,\infty) : X).$

The most intriguing properties of almost periodic vector-valued functions are collected in the following two theorems (in the case that $I={\mathbb R},$ these assertions are well-known in the existing literature; in the case that $I=[0,\infty),$ then these assertions can be deduced by using their validity in the case $I={\mathbb R}$ and the properties of extension mapping $E(\cdot);$
see \cite{Ko} for more details).

\begin{theorem}\label{svinja}
Let $f \in AP(I : X).$ Then the following holds:
\begin{itemize}
\item[(i)] $f\in BUC(I : X);$
\item[(ii)] if $g \in AP(I :X),$ $h \in AP(I :{\mathbb C}),$ $\alpha,\ \beta \in {\mathbb C},$ then $\alpha f + \beta g$ and $ hf \in AP(I :X);$
\item[(iii)] Bohr's transform of $f(\cdot),$
$$
P_{r}(f) := \lim_{t\rightarrow \infty}\frac{1}{t}\int^{t}_{0}e^{-irs}f(s)\, ds,
$$
exists for all $r\in {\mathbb R}$ and
$$
P_{r}(f) := \lim_{t\rightarrow \infty}\frac{1}{t}\int^{t+\alpha}_{\alpha}e^{-irs}f(s)\, ds
$$
for all $\alpha \in I,\ r\in {\mathbb R};$
\item[(iv)] if $P_{r}(f) = 0$ for all $r \in {\mathbb R},$ then $f(t) = 0$ for all $t \in I;$
\item[(v)] $\sigma(f):=\{r\in {\mathbb R} :  P_{r}(f) \neq 0\}$ is at most countable;
\item[(vi)] if $c_{0} \nsubseteq X,$ which means that $X$ does not contain an isomorphic copy of $c_{0},$  $I={\mathbb R}$
and $g(t) =\int^{t}_{0}f(s)\, ds$ ($t \in {\mathbb R}$) is bounded, then $g \in AP({\mathbb R} :X);$
\item[(vii)] if $(g_{n})_{n\in {\mathbb N}}$ is a sequence in $AP(I:X)$ and $(g_{n})_{n\in {\mathbb N}}$ converges uniformly to $g$, then
$g \in AP(I:X);$
\item[(viii)] if $I={\mathbb R}$ and $f^{\prime} \in BUC({\mathbb R}:X),$ then $f^{\prime} \in AP({\mathbb R} : X);$
\item[(ix)] (Spectral synthesis) $f\in \overline{span\{ e^{i\mu \cdot } x : \mu \in \sigma(f),\ x\in R(f)\}};$
\item[(x)] $R(f)$ is relatively compact in $X;$
\item[(xi)] we have
\begin{align*}
\|f\|_{\infty}=\sup_{t\geq t_{0}}\|f(t)\|,\quad t_{0}\in I.
\end{align*}
\end{itemize}
\end{theorem}

\begin{theorem}\label{vorw}
(Bochner's criterion) Let $f\in BUC({\mathbb R}: X).$
Then $f(\cdot)$ is almost periodic iff for any sequence $(b_n)$ of numbers from ${\mathbb R}$ there exists a subsequence $(a_{n})$ of $(b_n)$
such that $(f(\cdot+a_{n}))$ converges in $BUC({\mathbb R}: X).$
\end{theorem}

Theorem \ref{vorw} has served S. Bochner to introduce the notion of an almost automorphic function, which slightly generalize the notion of an almost periodic function \cite{Bo}. For more details about almost periodic and almost automorphic solutions of abstract Volterra integro-differential equations, we refer the reader to the monographs by T. Diagana \cite{Di},
G.~M.~N'Gu\' er\' ekata \cite{Gu},
M. Kosti\' c \cite{Ko} and M. Levitan, V. V. Zhikov \cite{Le}.

By either $AP(\Lambda : X)$ or $AP_{\Lambda}(I:X),$ where $ \Lambda$ is a non-empty subset of $I,$ we denote the vector subspace of $AP(I:X)$ consisting of all functions $f\in AP(I:X)$ for which the inclusion $\sigma(f)\subseteq \Lambda$ holds good. It can be easily seen that
$AP(\Lambda : X)$ is a closed subspace of $AP(I:X)$ and therefore Banach space itself.

\section{Almost Anti-Periodic Functions}\label{kragujevac}

Assume that $I={\mathbb R}$ or $I=[0,\infty),$ as well as that $f : I \rightarrow X$ is continuous. Given $\epsilon>0,$ we call $\tau>0$ an $\epsilon$-antiperiod for $f(\cdot)$ iff
\begin{align}\label{r-basara}
\| f(t+\tau) +f(t) \| \leq \epsilon,\quad t\in I.
\end{align}
In what follows, by $\vartheta_{ap}(f,\epsilon)$ we denote the set of all $\epsilon$-antiperiods for $f(\cdot).$

We introduce the notion of an almost anti-periodic function as follows.

\begin{definition}\label{anti-periodic}
It is said that $f(\cdot)$ is almost anti-periodic iff for each $\epsilon>0$ the set $\vartheta_{ap}(f,\epsilon)$ is relatively dense in $I.$
\end{definition}

Suppose that $\tau>0$ is an $\epsilon$-antiperiod for $f(\cdot).$ Applying \eqref{r-basara} twice, we get that
\begin{align*}
\| f(t+2\tau) & -f(t)\| =\bigl\| [f(t+2\tau)+f(t+\tau)]-[f(t+\tau)+f(t)]\bigr\|
\\ & \leq \| f(t+2\tau) +f(t+\tau)\| +\|f(t+\tau) +f(t)\|\leq 2 \epsilon,\quad t\in I.
\end{align*}
Taking this inequality in account, we obtain almost immediately from elementary definitions that $f(\cdot)$ needs to be almost periodic. Further on, assume that
$f : I \rightarrow X$ is anti-periodic, i.e, there exists $\omega>0$ such that $f(t+\omega)=-f(t),$ $t\in I.$ Then we obtain inductively that $f(t+(2k+1) \omega)=-f(t),$ $k\in {\mathbb Z},$ $t\in I.$ Since the set $\{(2k+1)\omega : k\in {\mathbb Z}\}$ is relatively dense in $I,$ the above implies that $f(\cdot)$ is almost anti-periodic. Therefore, we have proved the following theorem:

\begin{theorem}\label{izgubio se}
\begin{itemize}
\item[(i)] Assume $f : I \rightarrow X$ is almost anti-periodic. Then $f : I \rightarrow X$ is almost periodic.
\item[(ii)] Assume $f : I \rightarrow X$ is anti-periodic. Then $f : I \rightarrow X$ is almost anti-periodic.
\end{itemize}
\end{theorem}

It is well known that any anti-periodic function $f : I \rightarrow X$ is periodic since, with the notation used above, we have that
$f(t+2k \omega)=f(t),$ $k\in {\mathbb Z} \setminus \{0\},$ $t\in I.$ But, the constant non-zero function is a simple example of a periodic function (therefore, almost periodic function) that is neither anti-periodic nor almost anti-periodic.

\begin{example}\label{op}
\begin{itemize}
\item[(i)] Consider the function $f(t):=\sin (\pi t) +\sin (\pi t \sqrt{2}),$ $t\in {\mathbb R}.$ This is an example of an almost anti-periodic function that is not a periodic function. This can be verified as it has been done by A. S. Besicovitch \cite[Introduction, p. ix]{Be}.
\item[(ii)] The function $g(t):=f(t)+5,$ $t\in {\mathbb R},$ where $f(\cdot)$ is defined as above, is almost periodic, not almost anti-periodic and not periodic.
\end{itemize}
\end{example}

We continue by noting the following simple facts. Let $f : I \rightarrow X$ be continuous, and let $\epsilon'>\epsilon>0.$ Then the following holds true:
\begin{itemize}
\item[(i)] $\vartheta_{ap}(f,\epsilon) \subseteq \vartheta_{ap}(f,\epsilon').$
\item[(ii)] If $I={\mathbb R}$ and \eqref{r-basara} holds with some $\tau>0,$ then \eqref{r-basara} holds with $-\tau.$
\item[(iii)] If $I={\mathbb R}$ and $\tau_{1},\ \tau_{2} \in \vartheta_{ap}(f,\epsilon), $ then $\tau_{1}\pm \tau_{2} \in \vartheta(f,\epsilon).$
\end{itemize}

Furthermore, the argumentation contained in the proofs of structural results of \cite[pp. 3-4]{Be} shows that the following holds:

\begin{theorem}\label{krew}
Let $f  : I\rightarrow X$ be almost anti-periodic. Then we have:
\begin{itemize}
\item[(i)] $cf(\cdot)$ is almost anti-periodic for any $c\in {\mathbb C}.$
\item[(ii)] If $X={\mathbb C}$ and $\inf_{x\in {\mathbb R}}|f(x)|=m>0,$ then $1/f(\cdot)$ is almost anti-periodic.
\item[(iii)] If $(g_{n}: I \rightarrow X)_{n\in {\mathbb N}}$ is a sequence of almost anti-periodic functions and $(g_{n})_{n\in {\mathbb N}}$ converges uniformly to a function $g: I \rightarrow X$, then
$g(\cdot)$ is almost anti-periodic.
\end{itemize}
\end{theorem}

Concerning products and sums of almost anti-periodic functions, the situation is much more complicated than for the usually examined class of almost periodic functions:

\begin{example}\label{prc}
\begin{itemize}
\item[(i)] The product of two scalar almost anti-periodic functions need not be almost anti-periodic. To see this, consider the functions $f_{1}(t)=f_{2}(t)=\cos t,$ $t\in {\mathbb R},$ which are clearly (almost) anti-periodic. Then $f_{1}(t)\cdot f_{2}(t)=\cos^{2}t,$ $t\in {\mathbb R},$ $\cos^{2}(t+\tau) +\cos^{2}t\geq \cos^{2}t,$  $\tau,\ t\in {\mathbb R}$ and therefore $\vartheta_{ap}(f_{1}\cdot f_{2},\epsilon)=\emptyset$ for any $\epsilon \in (0,1).$
\item[(ii)] The sum of two scalar almost anti-periodic functions need not be almost anti-periodic, so that the almost anti-periodic functions do not form a vector space. To see this, consider the functions $f_{1}(t)=2^{-1}\cos 4t $ and $f_{2}(t)=2 \cos 2t,$ $t\in {\mathbb R},$ which are clearly (almost) anti-periodic. Then
$$
f_{1}(t)+f_{2}(t)=4\cos^{4}t-\frac{3}{2},\quad t\in {\mathbb R}.
$$
Asssume that $f_{1}+f_{2}$ is almost anti-periodic. Then the above identity implies that the function $t\mapsto 8\cos^{4}t-3,$ $t\in {\mathbb R}$ is almost anti-periodic, as well. This, in particular, yields that for any $\epsilon \in (0,1)$ we can find $\tau \in {\mathbb R}$
such that
$$
\bigl | 8\cos^{4}(t+\tau) +8\cos^{4}t -6\bigr| \leq \epsilon,\quad t\in {\mathbb R}.
$$
Plugging $t=\pi,$ we get that $8\cos^{4}\tau +2 \leq \epsilon,$ which is a contradiction. Finally, we would like to point out that there exists a large number of much simpler examples which can be used for verification of the statement clarified in this part; for example, the interested reader can easily check that the function $t\mapsto  \cos t +\cos 2t,$ $t\in {\mathbb R}$ is not almost anti-periodic.
\end{itemize}
\end{example}

Assume that $f : I\rightarrow X$ is almost anti-periodic. Then it can be easily seen that $f(\cdot+a)$ and $f(b\, \cdot)$ are likewise almost anti-periodic,
where $a\in I$ and $b\in I \setminus \{0\}.$

Denote now by $ANP_{0}(I : X)$ the linear span of almost anti-periodic functions $I \mapsto X.$ By Theorem \ref{izgubio se}(i), $ANP_{0}(I : X)$ is a linear subspace of $AP(I : X).$ Let $ANP(I : X)$ be the linear closure of $ANP_{0}(I : X)$ in $AP(I : X).$ Then, clearly,  $ANP(I : X)$ is a Banach space. Furthermore, we have the following result:

\begin{theorem}\label{ghf-prc}
$ANP(I : X)=AP_{{\mathbb R} \setminus \{0\}}(I : X).$
\end{theorem}

\begin{proof}
Since the mapping $E : AP([0,\infty) : X) \rightarrow AP({\mathbb R} : X)$ is a linear surjective isometry, it suffices to consider the case in which $I={\mathbb R}.$ Assume first that
$f\in AP_{{\mathbb R} \setminus \{0\}}(I : X).$
By spectral synthesis (see Theorem \ref{svinja}(ix)), we have that
$$
f\in \overline{span\{e^{i\mu \cdot} x : \mu \in \sigma(f),\ x\in R(f)\}},
$$
where the closure is taken in the space $AP({\mathbb R} : X).$
Since $\sigma(f) \subseteq {\mathbb R} \setminus \{0\}$ and the function $t\mapsto e^{i\mu t},$ $t\in {\mathbb R}$ ($\mu \in {\mathbb R} \setminus \{0\}$) is anti-periodic, we have that $span\{e^{i\mu \cdot} x : \mu \in \sigma(f),\ x\in R(f)\}\subseteq ANP_{0}({\mathbb R} : X).$ Hence, $f\in ANP({\mathbb R} : X).$ The converse statement immediately follows if we prove that, for any fixed function  $f\in ANP({\mathbb R} : X),$ we have that $P_{0}(f)=0,$ i.e.,
\begin{align}\label{spektar}
\lim_{t\rightarrow \infty}\frac{1}{t}\int^{t}_{0}f(s)\, ds =0.
\end{align}
By almost periodicity of $f(\cdot),$ the limit in \eqref{spektar} exists.
Hence, it is enough to show that for any
given number $\epsilon>0$ we can find a sequence $(\omega_{n})_{n\in {\mathbb N}}$ of positive reals such that $\lim_{n\rightarrow \infty}\omega_{n}=\infty$ and
\begin{align}\label{mavari}
\Biggl \| \frac{1}{2\omega_{n}}\int^{2\omega_{n}}_{0}f(s)\, ds \Biggr \| \leq \epsilon/2,\quad n\in {\mathbb N}.
\end{align}
By definition of almost anti-periodicity, we have the existence of a number $l>0$ such that any interval $I_{n}=[nl,(n+1)l]$ ($n\in {\mathbb N}$) contains a number $\omega_{n}$ that is anti-period for $f(\cdot).$ The validity of \eqref{mavari} is a consequence of the following computation:
\begin{align*}
\Biggl \| & \int^{2\omega_{n}}_{0}f(s)\, ds \Biggr \|
\\& =\Biggl \| \int^{\omega_{n}}_{0}f(s)\, ds + \int^{2\omega_{n}}_{\omega_{n}}f(s)\, ds \Biggr \|
\\ & =\Biggl \| \int^{\omega_{n}}_{0}\bigl[f(s)+f(s+\omega_{n})\bigr]\, ds \Biggr \|
\\ & \leq \int^{\omega_{n}}_{0}\bigl \| f(s)+f(s+\omega_{n})\bigr \|\, ds\leq \epsilon \omega_{n},\quad n\in {\mathbb N},
\end{align*}
finishing the proof of theorem.
\end{proof}

Let $f\in AP(I:X)$ and $\emptyset \neq \Lambda \subseteq {\mathbb R}.$ Since
$$
\sigma(f)\subseteq \Lambda \mbox{ iff } P_{r}(f)=0,\ r\in {\mathbb R} \setminus \Lambda \mbox{ iff } P_{0}\bigl( e^{-ir\cdot}f(\cdot)\bigr)=0,\ r\in {\mathbb R} \setminus \Lambda
$$
we have the following corollary of Theorem \ref{ghf-prc} (see also \cite[Corollary 4.5.9]{Ar}):

\begin{corollary}\label{ghf-prcc}
Let $f\in AP(I:X)$ and $\emptyset \neq \Lambda \subseteq {\mathbb R}.$ Then $f\in AP_{\Lambda}(I:X)$ iff $e^{-ir\cdot}f(\cdot)\in
ANP(I : X)$ for all $r\in {\mathbb R} \setminus \Lambda.$
\end{corollary}

Further on, Theorem \ref{ghf-prc} combined with the obvious equality $\sigma(Ef)=\sigma(f)$ immediately implies that the unique ANP extension of a function $f\in ANP([0,\infty) : X)$ to the whole real axis is $Ef(\cdot).$ As the next proposition shows, this also holds for almost anti-periodic functions:

\begin{proposition}\label{batty}
Suppose that $f : [0,\infty) \rightarrow X$ is almost anti-periodic. Then $Ef: {\mathbb R} \rightarrow X$ is a unique almost anti-periodic extension
of $f(\cdot)$ to the whole real axis.
\end{proposition}

\begin{proof}
The uniqueness of an almost anti-periodic extension of $f(\cdot)$ follows from the uniqueness of an almost periodic extension of $f(\cdot).$
It remains to be proved that $Ef: {\mathbb R} \rightarrow X$ is almost anti-periodic. To see this, let $\epsilon>0$ be given. Then there exists $l>0$ such that any interval $I\subseteq [0,\infty)$ of length $l$ contains a number $\tau \in I$ such that $\|f(s+\tau)+f(s)\|\leq \epsilon,$
$s\geq 0.$
We only need to prove that any interval $I\subseteq {\mathbb R}$ of length $2l$ contains a number $\tau \in I$ such that
$$
\bigl\|[Ef](t+\tau)+[Ef](t)\bigr\|=\bigl\| [W(t+\tau)f +W(t)f](0)\bigr\| \leq \epsilon,\quad
t\in {\mathbb R}.
$$
If $I\subseteq [0,\infty),$ then the situation is completely clear. Suppose now that
$I\subseteq (-\infty,0].$ Then $-I\subseteq [0,\infty)$ and there exists a number $-\tau \in -I$ such that $\sup_{s\geq 0}\| f(s-\tau)+f(s) \|\leq \epsilon.$ Then the conclusion follows from the computation
\begin{align*}
\| [W(t+\tau)f &+W(t)f](0)\|\leq \| W(t+\tau) f+W(t)f\|_{L^{\infty}([0,\infty))}
\\ & \leq \|W(t+\tau)\|_{L^{\infty}([0,\infty))} \|W(-\tau)f+f\|_{L^{\infty}([0,\infty))}
\\ & =\sup_{s\geq 0}\| f(s-\tau)+f(s) \| \leq \epsilon,\quad t\in {\mathbb R}.
\end{align*}
Finally, if $I=I_{1}\cup I_{2},$ where $I_{1}=[a,0]$ ($a<0$) and $I_{2}=[0,b]$ ($b>0$), then $|a|\geq l$ or $b\geq l.$ In the case that
$|a|\geq l,$ then the conclusion follows similarly as in the previously considered case. If $b\geq l,$ then the conclusion follows from the computation
 \begin{align*}
\| [W(t+\tau)f &+W(t)f](0)\|\leq \| W(t+\tau) f+W(t)f\|_{L^{\infty}([0,\infty))}
\\ & \leq \|W(t)\|_{L^{\infty}([0,\infty))} \|W(\tau)f+f\|_{L^{\infty}([0,\infty))}
\\ & =\sup_{s\geq 0}\| f(s+\tau)+f(s) \| \leq \epsilon,\quad t\in {\mathbb R},
\end{align*}
where $\tau \in I_{2}$ is an $\epsilon$-antiperiod of $f(\cdot).$
\end{proof}

For various generalizations of almost periodic functions, we refer the reader to \cite{Ko}. In the following definition, we introduce the notion of a Stepanov almost anti-periodic function.

\begin{definition}\label{stepa-anti-al}
Let $1\leq p<\infty,$ and let $f\in L_{loc}^{p}(I : X).$ Then we say that $f(\cdot)$ is Stepanov $p$-almost anti-periodic function, $S^{p}$-almost anti-periodic shortly, iff the function $\hat{f}:  I \rightarrow L^{p}([0,1]:X),$ defined by
$$
\hat{f}(t)(s):=f(t+s),\quad t\in I,\ s\in [0,1],
$$
is almost anti-periodic.
\end{definition}

It can be easily seen that any almost anti-periodic function needs to be $S^{p}$-almost anti-periodic, as well as that any $S^{p}$-almost anti-periodic function has to be $S^{p}$-almost periodic ($1\leq p<\infty$).

\section{Almost Anti-Periodic Properties of Convolution Products}\label{prckovic-entai}

Since almost anti-periodic functions do not form a vector space, we will focus our attention here to the almost anti-periodic properties of finite and infinite convolution product, which is undoubtedly a safe and sound way for providing certain applications to abstract PDEs.

\begin{proposition}\label{cuj-rad}
Suppose that $1\leq p <\infty,$ $1/p +1/q=1$
and $(R(t))_{t> 0}\subseteq L(X)$ is a strongly continuous operator family satisfying that $M:=\sum_{k=0}^{\infty}\|R(\cdot)\|_{L^{q}[k,k+1]}<\infty .$ If $g : {\mathbb R} \rightarrow X$ is $S^{p}$-almost anti-periodic, then the function $G(\cdot),$ given by
\begin{align}\label{wer}
G(t):=\int^{t}_{-\infty}R(t-s)g(s)\, ds,\quad t \in {\mathbb R},
\end{align}
is well-defined and almost anti-periodic.
\end{proposition}

\begin{proof}
It can be easily seen that, for every $ t\in {\mathbb R},$ we have $G(t)=\int^{\infty}_{0}R(s)g(t-s)\, ds.$ Since $g(\cdot)$ is $S^{p}$-almost periodic, we can apply \cite[Proposition 2.11]{Kos} in order to see that $G(\cdot)$ is well-defined and almost periodic.
It remains to be proved that $G(\cdot)$ is almost anti-periodic.
Let a number $\epsilon>0$ be given in advance.
Then we can find a finite number $l>0$
such that any subinterval $I$ of ${\mathbb R}$ of length $l$ contains a number
$\tau \in I$ such that
$
\int^{t+1}_{t} \| g(s+\tau)+g(s) \|^{p}\, ds \leq \epsilon^{p}, $ $ t\in {\mathbb R}.
$ Applying H\"older inequality and this estimate, similarly as in the proof of above-mentioned proposition, we get that
\begin{align*}
\| G(t+\tau)& +G(t)\|
\\ & \leq \int^{\infty}_{0}\|R(r) \|\cdot \| g(t+\tau -r)+g(t-r) \| \, dr
\\ & =\sum _{k=0}^{\infty} \int^{k+1}_{k}\|R(r) \| \cdot \| g(t+\tau -r)+g(t-r) \| \, dr
\\ & \leq \sum _{k=0}^{\infty} \|R(\cdot)\|_{L^{q}[k,k+1]}\Biggl(\int^{k+1}_{k}\| g(t+\tau -r)+g(t-r) \|^{p} \, dr\Biggr)^{1/p}
\\ & =\sum _{k=0}^{\infty} \|R(\cdot)\|_{L^{q}[k,k+1]} \Biggl(\int_{t-k-1}^{t-k}\| g(s+\tau)+g(s) \|^{p} \, ds\Biggr)^{1/p}
\\ & \leq \sum _{k=0}^{\infty} \|R(\cdot)\|_{L^{q}[k,k+1]} \epsilon = M \epsilon,\quad t\in {\mathbb R},
\end{align*}
which clearly implies that the set of all $\epsilon$-antiperiods of $G(\cdot)$ is relatively dense in ${\mathbb R}$.
\end{proof}

In order to relax our exposition, we shall introduce the notion of an asymptotically ($S^p$-)almost anti-periodic
function in the following way (cf. also \cite[Lemma 1.1]{He}):

\begin{definition}\label{stewpa}
\begin{itemize}
\item[(i)] Let $f\in C_{b}([0,\infty) : X).$ Then we say that $f(\cdot)$ is asymptotically almost anti-periodic iff
there are two locally functions $g: {\mathbb R} \rightarrow X$ and
$q: [0,\infty)\rightarrow X$ satisfying the following conditions:
\begin{itemize}
\item[(a)] $g$ is almost anti-periodic,
\item[(b)] $q$ belongs to the class $C_{0}([0,\infty) :X),$
\item[(c)] $f(t)=g(t)+q(t)$ for all $t\geq 0.$
\end{itemize}
\item[(ii)]
Let $1\leq p<\infty,$ and let $f\in L_{loc}^{p}([0,\infty) : X).$ Then we say that $f(\cdot)$ is asymptotically Stepanov $p$-almost anti-periodic, asymptotically $S^{p}$-almost anti-periodic shortly, iff
there are two locally $p$-integrable functions $g: {\mathbb R} \rightarrow X$ and
$q: [0,\infty)\rightarrow X$ satisfying the following conditions:
\begin{itemize}
\item[(a)] $g$ is $S^p$-almost anti-periodic,
\item[(b)] $\hat{q}$ belongs to the class $C_{0}([0,\infty) : L^{p}([0,1]:X)),$
\item[(c)] $f(t)=g(t)+q(t)$ for all $t\geq 0.$
\end{itemize}
\end{itemize}
\end{definition}

Keeping in mind Proposition \ref{cuj-rad} and the proof of \cite[Propostion 2.13]{Kos}, we can simply clarify the following result:

\begin{proposition}\label{stewpa-wqer}
Suppose that $1\leq p <\infty,$ $1/p +1/q=1$
and $(R(t))_{t> 0}\subseteq L(X)$ is a strongly continuous operator family satisfying that, for every
$s\geq 0,$ we have that
$$
m_{s}:=\sum_{k=0}^{\infty}\|R(\cdot)\|_{L^{q}[s+k,s+k+1]}<\infty .
$$
Suppose, further, that $f : [0,\infty) \rightarrow X$ is asymptotically $S^{p}$-almost anti-periodic as well as that
the locally $p$-integrable functions $g: {\mathbb R} \rightarrow X,$
$q: [0,\infty)\rightarrow X$ satisfy the conditions from \emph{Definition \ref{stewpa}(ii)}. Let there exist a finite number $M >0$ such that
the following holds:
\begin{itemize}
\item[(i)] $ \lim_{t\rightarrow +\infty}\int^{t+1}_{t}\bigl[\int_{M}^{s}\|R(r)\| \|q(s-r)\| \, dr\bigr]^{p}\, ds=0.$
\item[(ii)] $\lim_{t\rightarrow +\infty}\int^{t+1}_{t}m_{s}^{p}\, ds=0.$
\end{itemize}
Then the function $H(\cdot),$ given by
\begin{align*}
H(t):=\int^{t}_{0}R(t-s)f(s)\, ds,\quad t\geq 0,
\end{align*}
is well-defined, bounded and asymptotically $S^{p}$-almost anti-periodic.
\end{proposition}

Before providing some applications, we want to note that our conclusions from \cite[Remark 2.14]{Kos} and \cite[Proposition 2.7.5]{Ko} can be reformulated for asymptotical almost anti-periodicity.

It is clear that we can apply results from this section in the study of existence and uniqueness of almost anti-periodic solutions of fractional Cauchy
inclusion
\begin{align*}
D_{t,+}^{\gamma}u(t)\in {\mathcal A}u(t)+f(t),\ t\in {\mathbb R},
\end{align*}
where $D_{t,+}^{\gamma}$ denotes the Riemann-Liouville fractional derivative of order $\gamma \in (0,1),$ and
$f : {\mathbb R} \rightarrow X$ satisfies certain properties, and ${\mathcal A}$ is a closed multivalued linear operator (see \cite{Fa} for the notion).
Furthermore, we can analyze the existence and uniqueness of asymptotically ($S^{p}$-) almost anti-periodic solutions of fractional Cauchy inclusion
\[
\hbox{(DFP)}_{f,\gamma} : \left\{
\begin{array}{l}
{\mathbf D}_{t}^{\gamma}u(t)\in {\mathcal A}u(t)+f(t),\ t\geq 0,\\
\quad u(0)=x_{0},
\end{array}
\right.
\]
where ${\mathbf D}_{t}^{\gamma}$ denotes the Caputo fractional derivative of order $\gamma \in (0,1],$ $x_{0}\in X$ and
$f : [0,\infty) \rightarrow X,$ satisfies certain properties, and ${\mathcal A}$ is a closed multivalued linear operator  (cf. \cite{Ko} for more details). Arguing so, we can analyze the existence and uniqueness of (asymptotically $S^{p}$-) almost anti-periodic solutions of the fractional Poisson heat equations
\[\left\{
\begin{array}{l}
D_{t,+}^{\gamma}[m(x)v(t,x)]=(\Delta -b )v(t,x) +f(t,x),\quad t\in {\mathbb R},\ x\in {\Omega};\\
v(t,x)=0,\quad (t,x)\in [0,\infty) \times \partial \Omega ,\\
\end{array}
\right.
\]
and
\[\left\{
\begin{array}{l}
{\mathbf D}_{t}^{\gamma}[m(x)v(t,x)]=(\Delta -b )v(t,x) +f(t,x),\quad t\geq 0,\ x\in {\Omega};\\
v(t,x)=0,\quad (t,x)\in [0,\infty) \times \partial \Omega ,\\
 m(x)v(0,x)=u_{0}(x),\quad x\in {\Omega},
\end{array}
\right.
\]
in the space $X:=L^{p}(\Omega),$ where $\Omega$ is a bounded domain in ${\mathbb R}^{n},$ $b>0,$ $m(x)\geq 0$ a.e. $x\in \Omega$, $m\in L^{\infty}(\Omega),$ $\gamma \in (0,1)$ and $1<p<\infty ;$ see \cite{Fa} and \cite{Ko} for further information in this direction.

For some other references regarding the existence and uniqueness of anti-periodic and Bloch periodic solutions of certain classes of abstract Volterra integro-differential equations, we refer the reader to
\cite{Ch}, \cite{Dim}, \cite{Guv}-\cite{Mo} and \cite{Li}-\cite{Liu}.

\vspace{0.1cm}

\noindent{\bf Acknowledgment:} This research is partially supported by grant 174024 of Ministry
of Science and Technological Development, Republic of Serbia.


\end{document}